\documentclass[12pt]{extarticle}
\usepackage[english]{babel}
\usepackage{graphicx}
\usepackage{framed}
\usepackage[normalem]{ulem}
\usepackage{amsmath}
\usepackage{url}
\usepackage[percent]{overpic}
\usepackage{mathtools}
\usepackage{tikz}
\usepackage{caption}
\usepackage{subfigure}
\usepackage{scalerel}
\usepackage{stackengine,wasysym}

\usepackage{amsthm}
\usepackage{bbm}
\usepackage{xcolor}
\usepackage{indentfirst}
\usepackage{comment}
\usepackage{mathtools}
\usepackage{amssymb}
\usepackage{amsfonts}
\usepackage{enumerate}
\usepackage[utf8]{inputenc}
\usepackage[top=1 in,bottom=1in, left=1 in, right=1 in]{geometry}

\renewcommand{\emptyset}{\varnothing}

\newtheorem{innercustomgeneric}{\customgenericname}
\providecommand{\customgenericname}{}

\newtheorem{theorem}{Theorem}
\newtheorem{proposition}{Proposition}

\newtheorem{definition}{Definition}

\newtheorem{remark}{Remark}

\newtheorem{lemma}{Lemma}

\usepackage[affil-it]{authblk}
\begin{document}

\title{Symmetric periodic solutions in the generalized Sitnikov problem with homotopy methods}
\author{Carlos Barrera-Anzaldo%
  \thanks{Electronic address: \texttt{barrera@math.unipd.it}; corresponding author}}
\affil{Department of Mathematics ``Tullio Levi-Civita'' \\ Università degli Studi di Padova}
\author{Carlos García-Azpeitia%
  \thanks{Electronic address: \texttt{cgazpe@mym.iimas.unam.mx}}}
\affil{Mathematics and Mechanics Department, IIMAS \\ Universidad Nacional Autónoma de México}

\maketitle

\begin{abstract}
The paper investigates a generalization of the classical Sitnikov problem, concentrating on the movement of a satellite along the $Z$-axis as it interacts with $n$ primary bodies in periodic motion. It establishes the existence of an infinite number of even and anti-periodic solutions with increasing periods. The proof employs the Leray-Schauder degree theory to trace the critical points of action functionals, using a homotopy from solutions when the primary bodies are transformed into circular orbits.
\end{abstract}

\section{Introduction}

We examine a specific case of the restricted $(n+1)$-body problem in $\mathbb{R}^{3}$ where the primary bodies with positive masses $m_{1},\dots,m_{n}$ follow a periodic solution of the planar 
$n$-body problem. By choosing an appropriate coordinate system and rescaling space and time, we ensure that the primaries move in the $XY$-plane on a $\pi$-periodic path and the gravitational
$G$ is set to $1$. Additionally, we assume that the primaries move symmetrically, such that the $Z$-axis is an invariant set under the flow associated with the satellite's equations of motion. Under these conditions, the satellite's position is determined by its $z$-coordinate and satisfies the following non-autonomous differential equation:
\begin{equation}
\label{eqmoviintro}
    \ddot{z}=-\sum_{j=1}^{n}\frac{m_{j}z}{\left( \left\Vert q_{j}(t)\right\Vert ^{2}+z^{2}\right)^{3/2}}.
\end{equation}
In the previous equation, $q_{j}(t)$ denotes the position at time $t$ of the $j$-th body and satisfies $q_{j}(t+\pi)=q_{j}(t)$, and $\Vert\cdot\Vert$ represents the Euclidean norm of $\mathbb{R}^{2}$. The well-known Sitnikov problem (see \cite{Sitnikov}) can be derived from Eq. \eqref{eqmoviintro} by considering two primary bodies with equal mass moving along Keplerian elliptic orbits. Thus, \eqref{eqmoviintro} can be viewed as a generalization of the Sitnikov Problem.

It is important to clarify that by ``generalization'', we consider a broader range of possible planar configurations for the primary bodies. Several authors have proposed generalizations of the Sitnikov problem in this direction. In \cite{soulis, bountis, li}, the $n$ primaries with equal masses rotate with a constant angular velocity around the origin. In \cite{pustylnikov,rivera1}, the $n$ primaries with equal masses follow Keplerian ellipses. In these works, the bodies are positioned at the vertices of a regular $n$-polygon. On the other hand, \cite{marchesin,beltritti1} explore motions where the primaries do not form regular polygons. Specifically, \cite{beltritti1} extends the model from \cite{rivera1} by considering homographic motions that preserve an admissible planar central configuration at all times (see Definitions 1 and 2 of \cite{beltritti1}). Our paper generalizes the previously described cases and encompasses even more solutions for the primaries. For example, our model includes the well-known Super-eight choreography as a special case. To the best of our knowledge, our model encompasses the most general configurations for the primary bodies. Other works extend the Sitnikov problem in different directions. In \cite{pandey}, the authors build upon the study in \cite{soulis} by considering oblate primaries. In \cite{lacomba}, the Sitnikov problem is extended by embedding it in $\mathbb{R}^{4}$. More recently, \cite{rivera2} presents a model where $2n$ primaries move according to a periodic Hip-Hop solution of the spatial $2n$-body problem.  

We will prove the existence of an infinite number of symmetric periodic solutions of Eq. \eqref{eqmoviintro}. More precisely, given any $\mathfrak{q}\in\mathbb{Z}^{+}$ sufficiently large, there exists a finite number $2\pi\mathfrak{q}$-periodic solutions (depending on $\mathfrak{q}$) that satisfy
\begin{subequations}
\label{symsolutionsintro}
\begin{gather} 
    z(t+\pi\mathfrak{q})=-z(t) \label{antiper} \\
    z(-t)=z(t) \label{even}.
\end{gather}
\end{subequations}
for every $t\in\mathbb{R}$. Each solution will be characterized by its number of zeros, guaranteeing that the solutions are different. Functions exhibiting property \eqref{antiper} are referred to as ``anti-periodic'' in the literature.

Several authors have studied the existence of solutions with similar symmetry conditions in generalized Sitnikov problems. For example, \cite{rivera1} demonstrates the existence (or nonexistence) of even and periodic families of periodic solutions for $n$ primaries in elliptic Keplerian orbits for $2\leq n \leq 234$. In \cite{beltritti2, beltritti3}, the authors identified families of even and periodic solutions in the generalization of the Sitnikov problem proposed in \cite{beltritti1}, for all eccentricities within $[0,1[$. These works utilize a global continuation method described in \cite{Ortega}, applied to the zeros of a specific map dependent on one parameter (the eccentricity of the elliptic orbits of the primaries), and employ Brouwer degree theory.

The symmetry condition for the primaries is that they move forming groups of $d$-polygons of bodies with equal masses, which are invariant by simultaneous time reflections and a space reflection. A similar condition is discussed in Section 2 of \cite{bakker}. We further establish specific algebraic conditions on the masses and positions of the primaries (see Definition \ref{defsym}).

In this work, we implement a novel homotopy method. We define a homotopy $H_{j}(t,\lambda)$ for each primary body to transform its orbit into a circular orbit. This procedure converts Eq. \eqref{eqmoviintro} into a family of differential equations parameterized by $\lambda \in [0,1]$. For $\lambda=0$, we obtain the generalized circular Sitnikov problem studied in \cite{beltritti1}, while for $\lambda=1$, we recover Eq. \eqref{eqmoviintro}. We will search for periodic solutions of the family of differential equations by identifying critical points of the associated family of action functionals. That is, we consider the family
\begin{equation*}
\label{actionsitnikovintro}
    \mathcal{A}_{\lambda}(z) = \int_{0}^{2\pi\mathfrak{q}} \left[\frac{1}{2} \left(\partial_{t} z(t)\right)^{2} + \sum_{j=1}^{n} \frac{m_{j}}{\left[\left\Vert H_{j}(t;\lambda) \right\Vert^{2} + z^{2}\right]^{1/2}}\right] \, \text{d}t,
\end{equation*}
defined on an appropriate vector space of periodic paths and parameterized by $\lambda\in[0,1]$. The objective is to locate the critical points of $\mathcal{A}_{0}$ and extend these points along the homotopy to find the critical points of $\mathcal{A}_{1}$. 

Under suitable regularity conditions, the critical points of a functional correspond to the zeros of its gradient map. Therefore, we will employ a global continuation method of the zeros of $\nabla\mathcal{A}_{\lambda}$. Since the gradient map is defined on a space of periodic paths, we need to use the Leray-Schauder (LS) degree theory to perform the continuation. Intuitively, the LS degree is an algebraic count of the zeros of certain maps between normed (not necessarily finite-dimensional) spaces. This approach distinguishes our work from the methods in \cite{Ortega, rivera1, beltritti2}, where the map is defined in a finite-dimensional vector space and the Brouwer degree theory is sufficient.

The first step in our method is to search critical points for the case $\lambda=0$. This case corresponds to a conservative system with one degree of freedom. Thus, critical points can be explicitly determined through a phase portrait analysis, using the properties of the period function discussed in Section 5 of \cite{beltritti1} (see Proposition \ref{lambda0}). To calculate the Leray-Schauder degree and perform the continuation, it is essential to continue only critical points with an appropriate number of zeros.

Next, we will extend the critical points for the case $\lambda=0$ using the homotopy. By continuing from a critical point at $\lambda = 0$, we can encounter the following cases for the branch, as illustrated in Figure \ref{figcontinuation}.
\begin{enumerate}
    \item The branch tends to infinity. 
    \item The branch ends in the intersection with another branch.
    \item The branch ends in the trivial solution.
    \item The branch reaches up to $\lambda=1$. 
\end{enumerate}
We will select critical points for $\lambda= 0$ that correspond to Case 4. To eliminate Case 1, we will use an ``a priori'' bounds argument (see Proposition \ref{bounded}), adapting the proof of Proposition 5.1 from \cite{Ortega}. This method relies on comparing solutions to differential inequalities. Case 2 will be discarded by demonstrating that two branches emerging from different points at $\lambda=0$ can intersect only at the trivial solution, a result that follows from the uniqueness of solutions to a differential equation (see Proposition \ref{unicity}). Finally, to rule out Case 3, we will construct a neighborhood around the trivial solution where only critical points with a specific number of zeros can arrive. The existence of this neighborhood follows from Sturm-Liouville Theory (see Proposition \ref{neighborhoodaround0}).  

\begin{figure}[ht]
\vskip 0.5cm
\centering
\captionsetup{width=.7\linewidth}
\begin{overpic}[width=.6\textwidth]{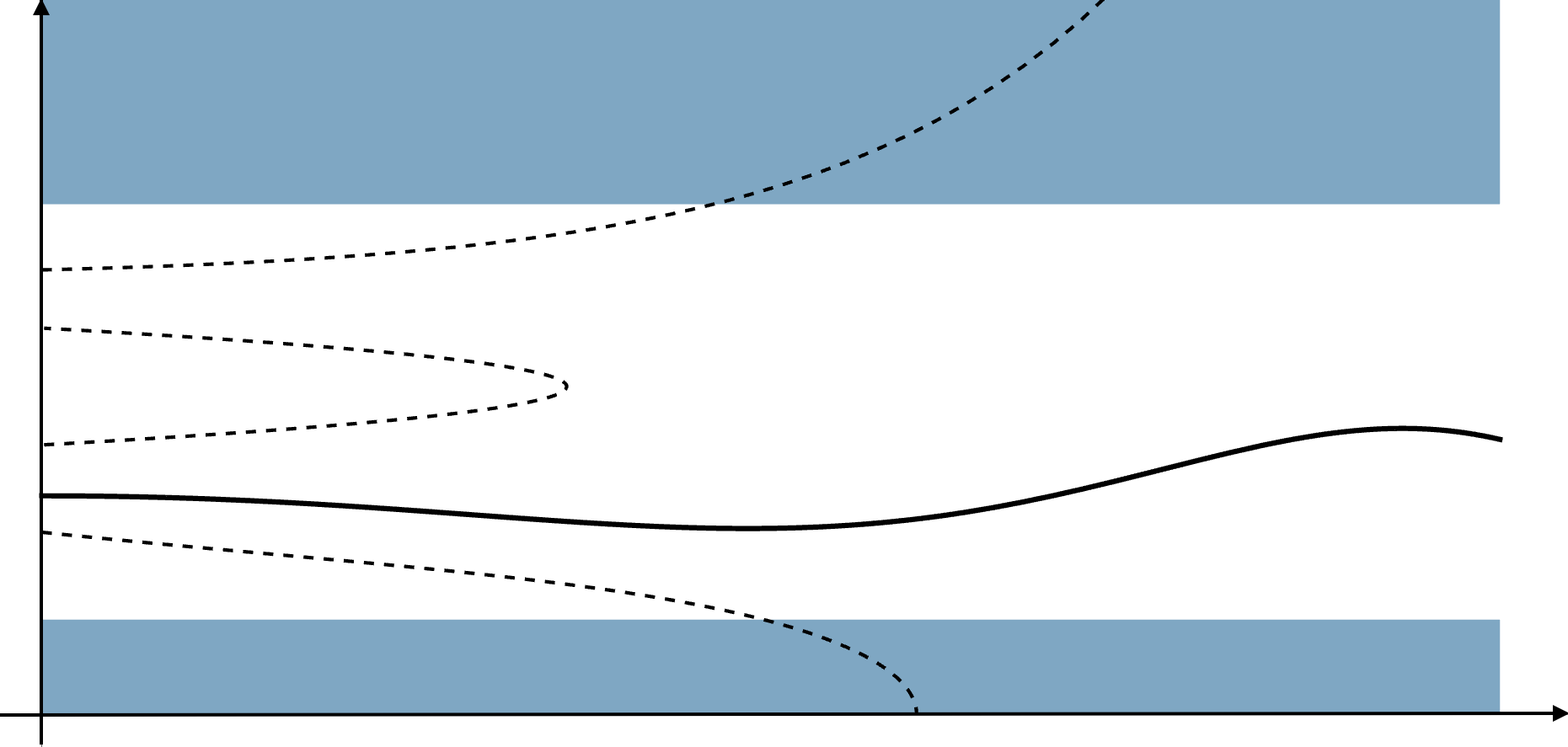}
    \put(-1,47){$\mathcal{Y}$}
    \put(101,1){$\lambda$}
    \put(94.7,-1.5){$1$}
    \put(0,-1.5){$0$}
    \put(94.7,18.7){$\bullet$}
\end{overpic}
\caption{Possible continuations of the branch of critical points at $\lambda=0$. 
Solid: admissible branch. Dashed: excluded alternatives. 
Blue: a-priori bound (Prop.~\ref{bounded}, upper) and neighborhood near the origin 
(Prop.~\ref{neighborhoodaround0}, lower).
}
\label{figcontinuation}
\end{figure}

The paper is organized as follows. In Sect.~\ref{sec:sec2}, we introduce $D_{d}$-symmetric planar configurations of the primaries and Eq.~\eqref{eqmoviintro}. In Theorem~\ref{main}, we state the existence of even and anti-periodic solutions. Sect.~\ref{sec:continuation} develops the continuation method necessary for the proof of Theorem~\ref{main}, recalling the required properties of the Leray--Schauder degree, defining the family of action functionals \eqref{actionsitnikovintro}, and describing the space of symmetric periodic paths. Theorem~\ref{main} is then proved using Propositions~\ref{lambda0}, \ref{bounded}, \ref{unicity}, and \ref{neighborhoodaround0}. Sect.~\ref{sec:conservative} analyzes the conservative case $\lambda=0$ and proves Proposition~\ref{lambda0} using period function properties from \cite{beltritti1}, while Sect.~\ref{sec:sec5} applies Sturm--Liouville theory to establish Proposition~\ref{neighborhoodaround0}.

\section{The generalized Sitnikov Problem}
\label{sec:sec2}

In this section, we introduce the generalized Sitnikov problem considered in this work. We specify the symmetry conditions on the primaries that guarantee the $Z$-axis is an invariant set under the flow of the satellite’s equations of motion. Finally, we state the main result of the paper.

\subsection{$D_{d}$-symmetric solutions of the planar $n$-body problem}
We consider the motion of $n$ bodies moving in the plane under their mutual gravitational interaction. Let $m_{j}>0$ denote the mass of the $j$-th body and $q_{j}(t)\in\mathbb{R}^{2}$ its position at time $t$, for $j=1,\dots,n$. We write $Q=(q_{1},\dots,q_{n})$ for a periodic solution of the $n$-body problem, that is, a solution of
\begin{equation}
\label{nbody}
\ddot{q}_{j}=-\displaystyle\sum_{\substack{i=1 \\ i\neq j}}^{n}m_{i}\frac{q_{i}-q_{j}}{\lVert q_{i}-q_{j}\rVert^{3}}, \qquad j=1,\dots,n.
\end{equation}
Here, $\lVert\cdot\rVert$ denotes the Euclidean norm in $\mathbb{R}^{2}$. These bodies will be referred to as the \textit{primaries}. After rescaling in space and time and making a translation in the plane, we can assume that $Q$ is $\pi$-periodic, $\sum_{j=1}^{n}m_{j}=1$, and
\begin{equation*}
    \displaystyle\sum_{j=1}^{n} m_{j}q_{j}(t) = 0,
\end{equation*}
for any $t\in\mathbb{R}/\pi\mathbb{Z}$. 

Let $q \in \mathbb{R}^{3}$ denote the position of a particle of infinitesimal mass (the satellite). Since the primaries lie in a plane, we may choose coordinates so that $q_{j}(t)=(x_{j}(t),y_{j}(t),0)$. In this setting, the equation of motion for $q=(x,y,z)$ takes the form
\begin{equation}
\label{generalposition}
    \ddot{q}=-\displaystyle\sum_{j=1}^{n}m_{j}\frac{q-q_{j}(t)}{\left[ \big(x-x_{j}(t)\big)^{2} + \big(y-y_{j}(t)\big)^{2} + z^2 \right]^{3/2}}
\end{equation}

In order to guarantee that the $Z$-axis is invariant under the flow of Eq.~\eqref{generalposition}, we impose suitable conditions on $Q$. Henceforth, $\mathbb{S}_{n}$ denotes the symmetric group on the set $\lbrace 1,\dots,n \rbrace$, and $J$ stands for the standard symplectic matrix, namely,
\begin{equation*}
    J=\begin{pmatrix} 0 & -1 \\ 1 & 0 \end{pmatrix}.
\end{equation*}
We define the numbers
\begin{equation}
\label{max}
\begin{gathered}
    \alpha_{j} = \displaystyle\min_{t\in\mathbb{R}/2\pi\mathbb{Z}} \Vert q_{j}(t)\Vert, \qquad \alpha=\displaystyle\sum_{j=1}^{n} \frac{m_{j}}{\alpha_{j}^{3}},\\
    \beta_{j} =\displaystyle\max_{t\in\mathbb{R}/2\pi\mathbb{Z}} \Vert q_{j}(t)\Vert, \qquad \beta=\displaystyle\sum_{j=1}^{n} \frac{m_{j}}{\beta_{j}^{3}}.
\end{gathered}
\end{equation}
The above numbers will be well-defined since the functions $q_{j}$ are continuous. The condition
\begin{equation*}
   \alpha_{\min} \coloneqq \displaystyle\min_{j=1,\dots,n} \alpha_{j} > 0,
\end{equation*}
is necessary to avoid collisions and we assume this holds hereafter. From now on, we denote by $D_{d}$ the dihedral group, that is, the group of symmetries of a regular $d$-gon, consisting of $d$ rotations and $d$ reflections.

\begin{definition}
\label{defsym}
A periodic solution $Q=(q_{1},\dots,q_{n})$ of the $n$-body problem \eqref{nbody} is said to be $D_{d}$-symmetric, with $d\geq 2$, if there exist generators $\zeta_{1},\zeta_{2}\in \mathbb{S}_{n}$ of a permutation subgroup isomorphic to the dihedral group $D_{d}$, together with an involution $R\in \mathsf{O}(2)$, such that the following conditions hold:  
\begin{equation}
\label{symmass}
    m_{\sigma(j)} = m_{j}, \qquad \forall\, \sigma \in D_{d},
\end{equation}
and
\begin{subequations}
\label{symmetry}
\begin{gather} 
    q_{\zeta_{1}(j)}(t) = e^{\frac{2\pi}{d}J}\, q_{j}(t), \label{sympoligon} \\
    q_{\zeta_{2}(j)}(t) = R\,q_{j}(-t). \label{symreverse}
\end{gather}
\end{subequations}
\end{definition}

Intuitively, a solution is $D_{d}$-symmetric when the bodies can be arranged into groups forming regular $d$-gons of equal masses (condition \eqref{sympoligon}), and, in addition, the configuration is invariant under the combined action of a time-reversal and a spatial reflection $R$ (condition \eqref{symreverse}). The polygonal symmetry imposed by \eqref{sympoligon} ensures that the $Z$-axis remains invariant under the flow of Eq.~\eqref{generalposition}, while the reversibility condition \eqref{symreverse} will be used to establish that the satellite equation is time-reversible.

\begin{lemma}
\label{lem:Zaxis}
Let $Q=(q_{1},\dots,q_{n})$ be a periodic $D_{d}$-symmetric solution of the $n$-body problem. Then the $Z$-axis is invariant under the flow of Eq.~\eqref{generalposition}.
\end{lemma}

\begin{proof}
Assume that the initial conditions of the satellite are 
\begin{equation*}
q(0)=(0,0,z_0), \qquad \dot{q}(0)=(0,0,\dot{z}_0).
\end{equation*}
From the first two components of Eq.~\eqref{generalposition}, we have
\begin{equation*}
(\ddot{x}(0),\ddot{y}(0)) = \sum_{j=1}^{n} \frac{m_j q_j(0)}{\big(\Vert q_j(0) \Vert^2 + z_0^2\big)^{3/2}}.
\end{equation*}
Since $Q$ is $D_d$-symmetric, there exists a permutation $\zeta \in \mathbb{S}_n$ such that 
\begin{equation*}
q_{\zeta(j)}(t) = e^{\frac{2\pi}{d}J} q_j(t), \qquad m_{\zeta(j)} = m_j.
\end{equation*}
Substituting this into the sum gives
\begin{equation*}
(\ddot{x}(0),\ddot{y}(0)) = \sum_{j=1}^{n} \frac{m_{\zeta(j)} q_{\zeta(j)}(0)}{\big(\Vert q_{\zeta(j)}(0) \Vert^2 + z_0^2\big)^{3/2}} 
= e^{\frac{2\pi}{d}J} \sum_{j=1}^{n} \frac{m_j q_j(0)}{\big(\Vert q_j(0) \Vert^2 + z_0^2\big)^{3/2}} 
= e^{\frac{2\pi}{d}J} (\ddot{x}(0),\ddot{y}(0)).
\end{equation*}
Since $d \ge 2$, the rotation $e^{\frac{2\pi}{d}J}$ is nontrivial. Hence,
\begin{equation*}
(\ddot{x}(0),\ddot{y}(0)) = (0,0).
\end{equation*}
Finally, since the vector field of Eq.~\eqref{generalposition} is $C^2$, the Existence and Uniqueness Theorem implies that
\begin{equation*}
(x(t),y(t)) = (0,0) \quad \text{for all } t.
\end{equation*}
Therefore, the $Z$-axis is invariant under the flow of Eq.~\eqref{generalposition}.
\end{proof}

There are several solutions to the planar $n$-body problem with this kind of symmetry in the literature. For example, the Super-Eight Choreography consists of 4 equal masses following the path illustrated in Figure \ref{super8plano}. The initial conditions are given in Eq. (19) of \cite{barrera}. Since this solution is a choreography, the symmetries of the initial conditions are preserved at any time. We can see from the initial conditions for the positions that
\begin{equation*}
    q_{3}=e^{\pi J}q_{1}(t); \qquad q_{4}=e^{\pi J}q_{2}(t).
\end{equation*}
Then, condition \eqref{sympoligon} is satisfied if $\zeta_{1}=(1 \ 3) \ (2 \ 4)$. Using the initial condition for the velocities and letting
\begin{equation*}
    R=\begin{pmatrix} 1 & 0 \\ 0 & -1 \end{pmatrix},
\end{equation*}
we obtain
\begin{equation*}
    q_{2}(t)=Rq_{1}(-t); \qquad q_{4}(t)=Rq_{3}(-t).
\end{equation*}
Then, condition \eqref{symreverse} is satisfied if $\zeta_{2}=(1 \ 2) \ (3 \ 4)$. Since the four bodies have equal masses, condition \eqref{symmass} follows immediately. Therefore, the Super-Eight choreography is a $D_{2}$-symmetric solution for $n=4$. The reader can find more $D_{d}$-symmetric solutions in \cite{calleja}.  

\begin{figure}[ht]
\vskip 0.5cm
\centering
\captionsetup{width=.7\linewidth}
\begin{overpic}[width=.6\textwidth]{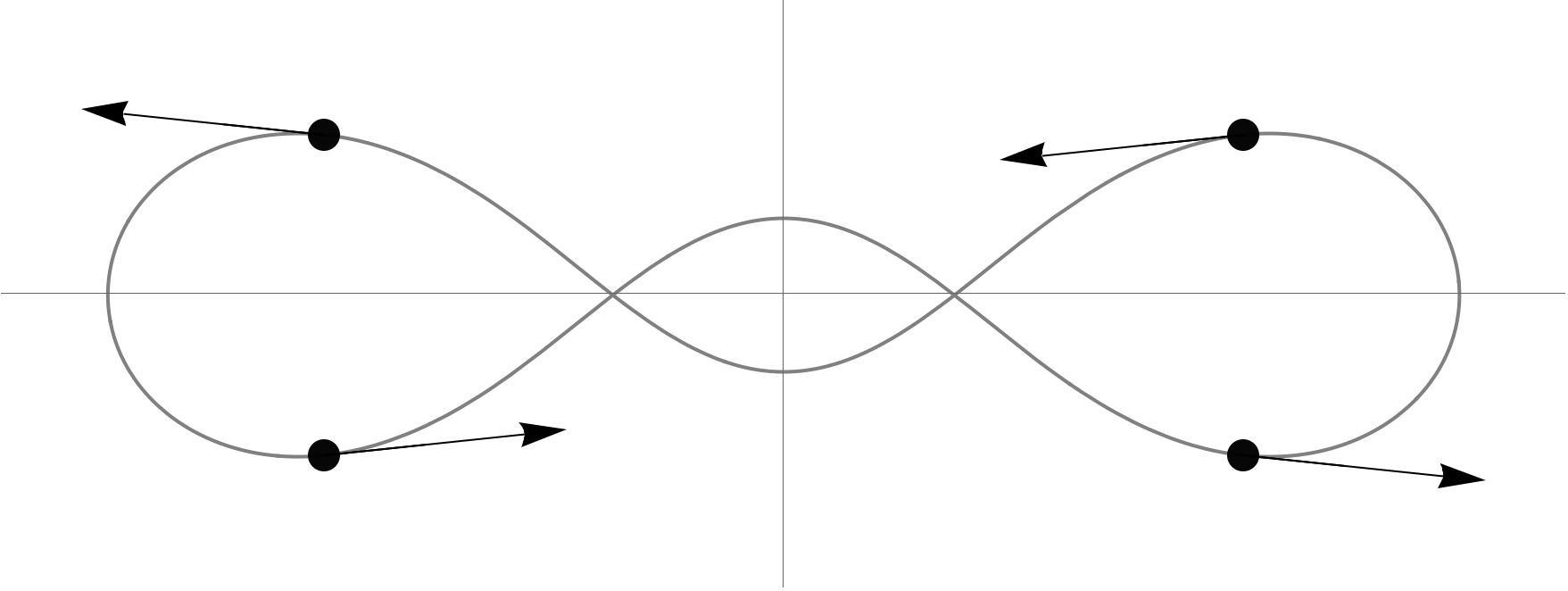}
    \put(19,31){$m_{1}$}
    \put(19,5){$m_{2}$}
    \put(78,5){$m_{3}$}
    \put(78,31){$m_{4}$}
    \put(78,31){$m_{4}$}
    \put(2,30){$\dot{q}_{1}$}
    \put(36.5,10){$\dot{q}_{2}$}
    \put(60,27){$\dot{q}_{4}$}
    \put(95.5,6.5){$\dot{q}_{3}$}
\end{overpic}
\caption{Super-Eight choreography. The figure depicts the positions and velocities of the bodies along the Super-Eight solution. The initial conditions satisfy the required $D_{2}$-symmetry, ensuring that the resulting motion is a $D_{2}$-symmetric solution of the planar $4$-body problem.}
\label{super8plano}
\end{figure}

\subsection{The $D_{d}$-symmetric Sitnikov problem}

When the primaries move in a $D_{d}$-symmetric solution, we can write the satellite's position as $q=(0,0,z)$. With this, the equation of motion of the coordinate $z$ becomes
\begin{equation}
\label{sitnikov}
    \ddot{z}=-\sum_{j=1}^{n}\frac{m_{j}z}{\left( \left\Vert q_{j}(t)\right \Vert ^{2}+z^{2}\right)^{3/2}}. 
\end{equation}
We refer to the previous equation as the \textit{$D_{d}$-symmetric Sitnikov problem}. Figure~\ref{figsitnikov} illustrates this model in the particular case where the primaries follow the Super-Eight choreography.

\begin{figure}[ht]
\vskip 0.5cm
\centering
\captionsetup{width=.7\linewidth}
\begin{overpic}[width=.6\textwidth]{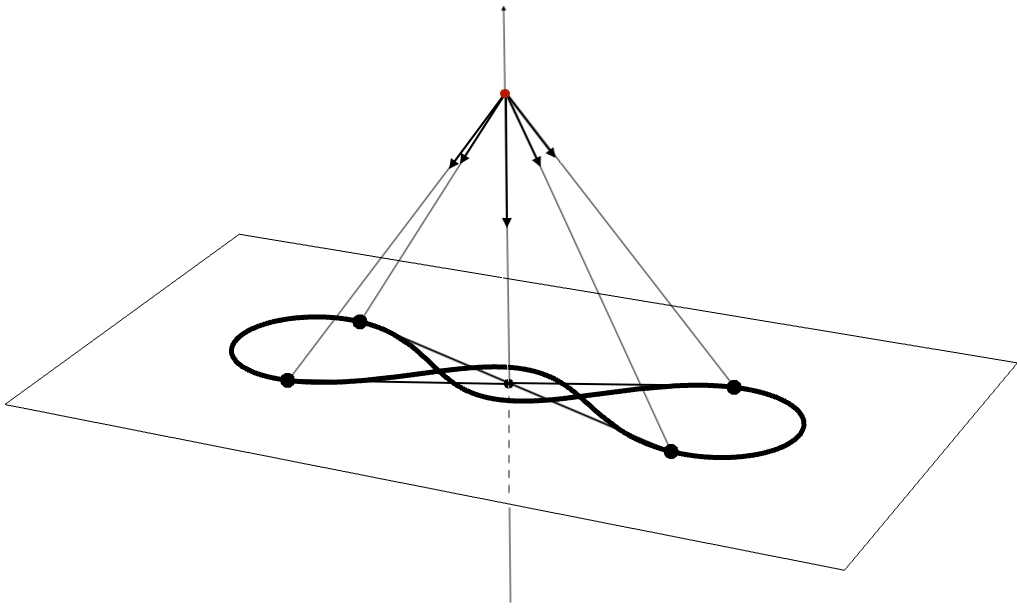}
    \put(34,26){\tiny $m_{1}$}
    \put(27,20.5){\tiny $m_{2}$}
    \put(71,19,5){\tiny $m_{4}$}
    \put(64,13.5){\tiny $m_{3}$}
    \put(51,51){\tiny $m=0$}
    \put(48,61){$Z$}
    \put(88,22){$\Pi_{XY}$}
    \put(51,35){$z$}
    \put(47.5,43){\tiny $\textbf{F}$}
    \put(40,28){\tiny\rotatebox{-20}{$\Vert q_{1}(t)\Vert$}}
    \put(53,19){\tiny\rotatebox{-22}{$\Vert q_{3}(t)\Vert$}}
    \put(35,20.3){\tiny\rotatebox{1}{$\Vert q_{2}(t)\Vert$}}
    \put(55,24){\tiny\rotatebox{1}{$\Vert q_{4}(t)\Vert$}}
\end{overpic}
\caption{Illustration of the $D_{d}$-symmetric Sitnikov problem with the primaries moving along the Super-Eight choreography. The symmetry of this configuration ensures that the gravitational force $\mathbf{F}$ acting on the satellite points only along the $Z$-axis, so the satellite’s trajectory is confined to the $Z$-axis.}
\label{figsitnikov}
\end{figure}

We aim to find subharmonic periodic solutions of Eq.~\eqref{sitnikov} exhibiting certain temporal symmetries. More specifically, we seek even solutions that are anti-periodic with period $2\pi \mathfrak{q}$, for many integers $\mathfrak{q} \in \mathbb{Z}^{+}$. Our main result is the following existence theorem.

\begin{theorem}
\label{main}
Consider $n$ bodies with masses $m_1,\dots,m_n$ moving in a $D_d$-symmetric, $\pi$-periodic solution of the planar $n$-body problem, and let $\beta$ be the constant defined in \eqref{max}. Then, for each integer $\mathfrak{q} > 1/\sqrt{\beta}$ and each integer $\mathfrak{p} \in \left\{ 1, \dots, \left[ \sqrt{\beta}\, \mathfrak{q} \right] \right\}$, there exists a $2\pi \mathfrak{q}$-periodic solution $z_{\mathfrak{p},\mathfrak{q}}$ of Eq.~\eqref{sitnikov} with the following properties:
\begin{enumerate}
    \item $z_{\mathfrak{p},\mathfrak{q}}(t)=-z_{\mathfrak{p},\mathfrak{q}}(t+\pi\mathfrak{q})$,
    \item $z_{\mathfrak{p},\mathfrak{q}}(t)=z_{\mathfrak{p},\mathfrak{q}}(-t)$,
    \item $z_{\mathfrak{p},\mathfrak{q}}$ has exactly $2\mathfrak{p}$ zeros in $[0,2\pi\mathfrak{q}]$.
\end{enumerate}
\end{theorem}

A direct consequence of Theorem~\ref{main} is the existence of infinitely many distinct periodic solutions of Eq.\eqref{sitnikov}, since each solution is characterized by the number of zeros $2\mathfrak{p}$. Moreover, property 1 of the theorem implies that each solution is anti-periodic with half-period $\pi \mathfrak{q}$, while property 2 guarantees that each solution is even in time. The proof of Theorem \ref{main} is deferred to Section \ref{sec:continuation}.

\section{A continuation method using Leray--Schauder degree}
\label{sec:continuation}

In this section, we prove Theorem \ref{main} employing a continuation method for mappings between infinite-dimensional spaces. We first express the motion of the $n$ primaries in polar coordinates. Specifically, for each body there exist functions $r_j:\mathbb{R}\to\mathbb{R}^+$ and $\theta_j:\mathbb{R}\to\mathbb{R}/2\pi\mathbb{Z}$, representing the modulus and argument, respectively, such that
\begin{equation*}
q_j(t) = r_j(t), e^{J \theta_j(t)}.
\end{equation*}
The existence of these functions is guaranteed by the fact that $\alpha>0$, which ensures that each primary stays away from the origin and thus polar coordinates are well-defined. 

For each body, we define the homotopy $H_{j}: \mathbb{R}\times[0,1]\to\mathbb{R}^{2}$ given by
\begin{equation}
\label{homotopy}
    H_{j}(t;\lambda)=[(1-\lambda )\beta_{j}+\lambda r_{j}(t)]e^{J\theta_{j}(t)}, \quad j=1,\dots,n. 
\end{equation}
where the number $\beta_{j}$ is given in \eqref{max}. Notice that $H_{j}$ is $\pi$-periodic in $t$. The previous homotopy defines the following family of differential equations parameterized by $\lambda$
\begin{equation}
\label{sitnikovhom}
    \ddot{z}=-\sum_{j=1}^{n}\frac{m_{j}z}{\left( \left\Vert H_{j}(t;\lambda)\right\Vert ^{2}+z^{2}\right)^{3/2}}.
\end{equation}
The case $\lambda = 1$ in Eq.~\eqref{sitnikovhom} corresponds to the $D_{d}$-symmetric Sitnikov Problem \eqref{sitnikov}. To find periodic solutions, we first identify solutions of Eq.~\eqref{sitnikovhom} when $\lambda = 0$ and then continue these solutions along the homotopy to $\lambda = 1$. We reformulate Eq. \eqref{sitnikovhom} as a variational problem, searching for solutions as critical points of an action functional.

\subsection{Variational setting}
\label{subsec:variational}

Fix $\mathfrak{q}\in \mathbb{Z}^{+}$. We consider the Sobolev space $H^{1}=H^{1}(\mathbb{R}/2\pi\mathfrak{q}\mathbb{Z},\mathbb{R})$, that is, the space of $2\pi\mathfrak{q}$-periodic functions with one (weak) derivative in $L^{2}(\mathbb{R}/2\pi\mathfrak{q}\mathbb{Z},\mathbb{R})$ and inner product
\begin{equation}
\label{innerproduct}
    \left\langle x,y \right\rangle_{H^{1}}=\displaystyle\int_{0}^{2\pi\mathfrak{q}}\left[ x(t)y(t) + \partial_{t}x(t)\partial_{t}y(t)\right] \ \mbox{d}t.
\end{equation} 
In the previous formula, $\partial_{t}x$ denotes the (weak) derivative of $x$.  For any $\lambda\in[0,1]$, let us consider the action functional $\mathcal{A} _{\lambda}:H^{1}\to\mathbb{R}$ given by
\begin{equation}
\label{actionsitnikov}
    \mathcal{A}_{\lambda}(z) = \displaystyle\int_{0}^{2\pi\mathfrak{q}} \left[\frac{1}{2}\partial_{t}z(t)^{2} -U_{\lambda}(t,z(t))\right] \ \mbox{d}t,
\end{equation}
where the potential energy $U_{\lambda}$ is given by
\begin{equation*}
    U_{\lambda}(t,z)=-\sum_{j=1}^{n}\frac{m_{j}}{\left[\left\Vert H_{j}(t;\lambda) \right\Vert^{2}+z^{2}\right]^{1/2}},
\end{equation*}
Notice that the function $U_{\lambda}$ is measurable for each $t$ and continuously differentiable in $z\in\mathbb{R}$ for every $t$. Then, by Theorem 1.4 from \cite{Mawhin}, the action functional \eqref{actionsitnikov} is continuously differentiable on $H^{1}$ for every $\lambda$. Also, by Corollary 1.1 from \cite{Mawhin}, every critical point $z\in H^{1}$ from $\mathcal{A}_{\lambda}$ is a $2\pi\mathfrak{q}$-periodic solution of Eq. \eqref{sitnikovhom}. 

To overcome the degeneracy of the action functional $\mathcal{A}_{\lambda}$ on the whole space $H^{1}$, we restrict our attention to a subspace of symmetric periodic paths. We consider the operators $\kappa_{1},\kappa_{2}: H^{1}\to H^{1}$ defined by
\begin{equation}
    (\kappa_{1} z)(t)=-z(t+\pi\mathfrak{q}), \qquad (\kappa_{2} z)(t)=z(-t).
\end{equation}

\begin{lemma}
Let us assume that $Q=(q_{1},\dots,q_{n})$ is a $D_{d}$-symmetric solution of the planar $n$-body problem. Then, the action functional $\mathcal{A}_{\lambda}$ given in \eqref{actionsitnikov} is invariant under the action of $\kappa_{1}$ and $\kappa_{2}$.
\end{lemma}
\begin{proof}
We need to verify that $\mathcal{A}_{\lambda}(\kappa_{1} z)=\mathcal{A}_{\lambda}(z)=\mathcal{A}_{\lambda}(\kappa_{2} z)$. First, we can notice that the potential $U_{\lambda}(t,z)$ is $\pi$-periodic in $t$ and even in $z$. By direct computation, we have that
\begin{equation*}
\begin{aligned}
    \mathcal{A}_{\lambda}(\kappa_{1} z)&=\displaystyle\int_{0}^{2\pi\mathfrak{q}}\left[ \frac{1}{2}\partial_{t}z(t+\pi\mathfrak{q})^{2}-U_{\lambda}(t,-z(t+\pi\mathfrak{q}))\right] \ \mbox{d}t \\
    &= \int_{\pi\mathfrak{q}} ^{3\pi\mathfrak{q}}\left[ \frac{1}{2}\partial_{t}z(\tau)^{2}-U_{\lambda}(\tau-\pi\mathfrak{q},-z(\tau))\right] \ \mbox{d}\tau \\
    &= \int_{\pi\mathfrak{q}} ^{3\pi\mathfrak{q}}\left[ \frac{1}{2}\partial_{t}z(\tau)^{2}-U_{\lambda}(\tau,z(\tau))\right] \ \mbox{d}\tau \\
    &= \mathcal{A}_{\lambda}(z).
\end{aligned}
\end{equation*}
Since $Q$ is $D_{d}$-symmetric there exist $\zeta\in\mathbb{S}^{n}$ and an involution $R\in\mathsf{O}(2)$ such that $q_{\zeta(j)}(t)=Rq_{j}(-t)$. Therefore,
\begin{equation*}
    \big\Vert H_{j}(-t;\lambda) \big\Vert = \big\Vert (1-\lambda)\beta_{j} + \lambda\Vert q_{j}(-t) \Vert \big\Vert = \big\Vert (1-\lambda)\beta_{j} + \lambda\Vert R q_{\zeta{j}}(t) \Vert \big\Vert = \big\Vert H_{\zeta(j)}(t;\lambda) \big\Vert.
\end{equation*}
Using the previous result and that $m_{j}=m_{\zeta(j)}$, we obtain
\begin{equation}
\label{reversibleintime}
    U_{\lambda}(-t,z)=-\displaystyle\sum_{j=1}^{n}\frac{m_{j}}{\big[ \Vert H_{j}(-t;\lambda) \Vert^{2}+z^{2}\big]^{1/2}}=-\displaystyle\sum_{j=1}^{n}\frac{m_{\zeta(j)}}{\big[ \Vert H_{\zeta(j)}(t;\lambda) \Vert^{2}+z^{2}\big]^{1/2}}=U_{\lambda}(t,z).
\end{equation}
Therefore, we have that
\begin{equation*}
\begin{aligned}
    \mathcal{A}_{\lambda}(\kappa_{2} z)&=\displaystyle\int_{0}^{2\pi\mathfrak{q}}\left[ \frac{1}{2}\partial_{t}z(-t)^{2}-U_{\lambda}(t,z(-t))\right] \ \mbox{d}t \\
    &=\displaystyle \int_{-2\pi\mathfrak{q}} ^{0}\left[ \frac{1}{2}\partial_{t}z(\tau)^{2}-U_{\lambda}(-\tau,z(\tau))\right] \ \mbox{d}\tau \\
    &=\displaystyle \int_{-2\pi\mathfrak{q}}^{0}\left[ \frac{1}{2}\partial_{t}z(\tau)^{2}-U_{\lambda}(\tau,z(\tau))\right] \ \mbox{d}\tau \\
    &= \mathcal{A}_{\lambda}(z),
\end{aligned}
\end{equation*}
and the result follows.
\end{proof}
We denote the set of fixed points under the action of $\kappa_{1}$ and $\kappa_{2}$ as 
\begin{equation*}
    \mathcal{Y}=\left\lbrace z\in H^{1} : z(t)=-z(t+\pi\mathfrak{q})=z(-t) \right\rbrace.
\end{equation*}
The set $\mathcal{Y}$ is a closed subspace of $H^{1}$. This implies that $\mathcal{Y}$ is a Hilbert space with the inner product \eqref{innerproduct}. Therefore, for each $\lambda\in[0,1]$ we can define the restricted functional
\begin{equation*}
    \mathcal{B}_{\lambda} : \mathcal{Y}\to\mathbb{R}, \qquad \mathcal{B}_{\lambda}(z)=\mathcal{A}_{\lambda}(z).
\end{equation*}

The functional $\mathcal{B}_{\lambda}$ inherits the same regularity properties as $\mathcal{A}_{\lambda}$. Recall that the first variation of $\mathcal{B}_{\lambda}$ at a point $z$ in the direction $w \in \mathcal{Y}$, often referred to as the directional derivative and denoted by $\delta \mathcal{B}_{\lambda}(z)[w]$, is defined by
\begin{equation*}
    \delta \mathcal{B}_{\lambda}(z)[w] = \displaystyle\lim_{s\to 0}\mathcal{B}_{\lambda}(z+sw).
\end{equation*}

\begin{proposition}
Let $\tilde{z}\in \mathcal{Y}$ be a critical point of $\mathcal{B}_{\lambda}$. Then $\tilde{z}$ is also a critical point of $\mathcal{A}_{\lambda}$. 
\end{proposition}

\begin{proof}
The claim follows directly from the Principle of Symmetric Criticality; see \cite{Palais} for details.
\end{proof}

Since $\mathcal{Y}$ is a Hilbert space, we can introduce the gradient map of the functional $\mathcal{B}_{\lambda}$ as the map that associates any $z\in \mathcal{Y}$ with the unique vector $v=\nabla\mathcal{B}_{\lambda}(z)$ that satisfies
\begin{equation}
\label{gradientdef}
    \left\langle v,w \right\rangle_{H^{1}}=\delta\mathcal{B}_{\lambda}(z)[w], \qquad w\in \mathcal{Y}.
\end{equation}

From here, we denote the set of linear operators defined over a Hilbert space $H$ by $\mathcal{L}(H)$.

\begin{lemma}
\label{fredholmap}
There exists a compact operator $K_{\lambda}\in \mathcal{L}(\mathcal{Y})$ such that the gradient map can be written as
\begin{equation}
\label{identidadmascompacto}
    \nabla\mathcal{B}_{\lambda}=I-K_{\lambda},
\end{equation}
where $I$ denotes the identity map.
\end{lemma}
\begin{proof}
For a fixed $z \in H^{1}$, let $u$ denote the unique solution of the equation
\begin{equation}
\label{eqgradient}
    \left\lbrace \begin{matrix} -\ddot{u} + u = \dfrac{\partial U_{\lambda}}{\partial z}(t,z(t))+z(t), \\ u\in H^{1}. \end{matrix} \right.
\end{equation}
Observe that if $z\in \mathcal{Y}$ in Eq.~\eqref{eqgradient}, then $u\in \mathcal{Y}$. Therefore, the map $z\mapsto u=K_{\lambda}(z)$ defines a linear operator on $\mathcal{Y}$. Using Eq. \eqref{gradientdef}, we have that
\begin{equation*}
\begin{aligned}
    \langle \nabla\mathcal{B}_{\lambda}(z), w \rangle_{H^{1}} & = \displaystyle\int_{0}^{2\pi\mathfrak{q}} \left[ \partial_{t}z(t)\partial_{t}w(t) - \dfrac{\partial U_{\lambda}}{\partial z}(t,z(t))w(t) \right] \mbox{d}t \\
    & = \displaystyle \int_{0}^{2\pi\mathfrak{q}} \left[ z(t)w(t) + \partial_{t}z(t)\partial_{t}w(t) - \dfrac{\partial U_{\lambda}}{\partial z}(t,z(t))w(t)-z(t)w(t) \right] \mbox{d}t \\
    & = \langle z , w \rangle_{H^{1}} - \langle K_{\lambda}(z) , w \rangle_{H^{1}}.
\end{aligned}
\end{equation*}
The previous formula is true for every $w\in \mathcal{Y}$. Therefore, formula \eqref{identidadmascompacto} holds.

We only need to prove that the operator $K_{\lambda}$ is compact. Since $u$ solves \eqref{eqgradient}, $u\in C^{2}=C^{2}(\mathbb{R}/2\pi\mathfrak{q}\mathbb{Z}, \mathbb{R})$ and $\Vert u \Vert_{C^{2}}$ is bounded by $\Vert z \Vert_{\infty}$. Let us recall that $\Vert z \Vert_{\infty}$ is dominated by $\Vert z \Vert_{\mathcal{Y}}$. Therefore $K_{\lambda}$ sends bounded sets in $\mathcal{Y}$ to bounded sets in $C^{2}$. Finally, since $C^{2}$ has a compact immersion in $H^{1}$, the map $K_{\lambda}$ is compact. 
\end{proof}

\begin{remark}
The decomposition \eqref{identidadmascompacto} shows that the gradient map 
$\nabla\mathcal{B}_{\lambda}=I-K_{\lambda}$ is a compact perturbation of the identity. 
In particular, it is a Fredholm operator of index zero. 
This property is essential, since it provides the functional framework in which the Leray–Schauder degree is well-defined.
\end{remark}

It can be seen that the gradient map $\nabla\mathcal{B}_{\lambda}$ is of class $C^{1}$. 
Consequently, its derivative at a point $z\in\mathcal{Y}$ naturally leads to the definition of the Hessian map, 
which is expressed in terms of the second variation $\delta^{2}\mathcal{A}_{\lambda}$. 
More precisely, for $u\in\mathcal{Y}$, the Hessian map assigns the unique vector $v=\operatorname{D}^{2}\mathcal{A}_{\lambda}(z)u$ such that
\begin{equation}
\label{definitionhessian}
    \langle v , w \rangle_{H^{1}} = \delta^{2}\mathcal{A}_{\lambda}(z)[u,w], \qquad w\in \mathcal{Y}.
\end{equation}

In analogy with the gradient representation obtained in Lemma \ref{fredholmap}, the next step is to examine the structure of the Hessian map. At a critical point, the Hessian map inherits a similarly convenient decomposition.

\begin{lemma}
\label{fredholmmapsecond}
If $z$ is a critical point of $\mathcal{B}_{\lambda}$, there exists a compact operator $L_{\lambda}(z)\in\mathcal{L}(\mathcal{Y})$ such that the Hessian map can be written as
\begin{equation}
\label{identidadmascompactosecond}
    \operatorname{D}^{2}\mathcal{B}_{\lambda}(z) = I - L_{\lambda}(z)
\end{equation}
\end{lemma}
\begin{proof}
Let $z\in\mathcal{Y}$ be a critical point of $\mathcal{B}_{\lambda}$. Given any $v\in H^{1}$, let $u$ be the unique solution of the equation
\begin{equation}
\label{eqhessian}
    \left\lbrace \begin{matrix} -\ddot{u} + u =\left[\dfrac{\partial^{2}U_{\lambda}}{\partial z^{2}}(t,z(t))+1\right]v(t), \\ u\in H^{1}. \end{matrix} \right.
\end{equation}
We can prove that $u\in\mathcal{Y}$ whenever $v\in\mathcal{Y}$. This implies that the map $v\mapsto u=L_{\lambda}(z)v$ is well-defined over $\mathcal{Y}$ and $L_{\lambda}(z)\in\mathcal{L}(\mathcal{Y})$. Using Eq. \eqref{definitionhessian}, we obtain 
\begin{equation*}
\begin{aligned}
    \langle \operatorname{D}^{2}\mathcal{B}_{\lambda}(z)v, w \rangle_{H^{1}} & = \displaystyle\int_{0}^{2\pi\mathfrak{q}} \left[ \partial_{t}v(t)\partial_{t}w(t) - \dfrac{\partial^{2}U_{\lambda}}{\partial z^{2}}(t,z(t))v(t)w(t) \right] \mbox{d}t \\
    & = \displaystyle\int_{0}^{2\pi\mathfrak{q}} \left[ v(t)w(t) + \partial_{t}v(t)\partial_{t}w(t) - \left(\dfrac{\partial^{2}U_{\lambda}}{\partial z^{2}}(t,z(t))+1\right)v(t)w(t) \right] \mbox{d}t \\
    & = \langle v , w \rangle_{H^{1}} - \langle L_{\lambda}(z)v , w \rangle_{H^{1}}.
\end{aligned}
\end{equation*}
The previous formula is true for every $w\in \mathcal{Y}$. Therefore, formula \eqref{identidadmascompactosecond} holds. Finally, using the same argument as in Lemma \eqref{fredholmap}, we can prove that $L_{\lambda}(z)$ is a compact operator when $z$ is a critical point of $\mathcal{B}_{\lambda}$, and the proof is complete. 
\end{proof}

\subsection{Properties of the Leray-Schauder degree}
Let us recall that a function $F: X\to Y$ between normed spaces is compact if it is continuous and $F(X)$ has a compact closure in $Y$. The LS degree is defined for mappings with the form $I-F$, where $I$ is the identity map and $F$ is compact. Intuitively, given any open and bounded set $U\subset X$, the LS degree $\deg_{\operatorname{LS}}[I-F, U, z]$ is an algebraic count of the number of solutions $x\in U$ of the equation
\begin{equation*}
    (I-F)(x)=z
\end{equation*}
For example, $\deg_{\operatorname{LS}}[I-F,U,z]=0$ when $z\not\in(I-F)(U)$. The LS degree $\deg_{\operatorname{LS}}[I-F, U, z]$ is constructed by approximating the completely continuous function $F$ by functions with range in a finite-dimensional subspace of $X$ containing $z$. 

We recall only the properties that will be used in the sequel. The complete construction and proofs of these properties can be found in \cite{mawhin2}. First, for a set $A \subset X \times [0,1]$ and $\lambda \in [0,1]$, we define
\begin{equation*}
A_\lambda = \{ x \in X : (x,\lambda) \in A \}.
\end{equation*}

The properties we will use are the following:
\begin{enumerate}
    \item \textit{Additivity.} If $U = U_1 \cup U_2$, where $U_1$ and $U_2$ are open and disjoint, and if 
    $z \notin (I-F)(\partial U_1) \cup (I-F)(\partial U_2)$, then
    \begin{equation*}
    \deg_{\operatorname{LS}}[I-F, U, z] = \deg_{\operatorname{LS}}[I-F, U_1, z] + \deg_{\operatorname{LS}}[I-F, U_2, z].
    \end{equation*}

    \item \textit{Existence.} If $\deg_{\operatorname{LS}}[I-F, U, z] \neq 0$, then $z \in (I-F)(U)$.

    \item \textit{Homotopy invariance.} Let $\Omega \subset X \times [0,1]$ be open and bounded, and let $F: \bar{\Omega} \to X$ be compact. If 
    $x - F(x, \lambda) \neq z$ for all $(x, \lambda) \in \partial \Omega$, then
    \begin{equation*}
    \deg_{\operatorname{LS}}[I-F(\cdot, \lambda), \Omega_\lambda, z] \text{ is independent of } \lambda.
    \end{equation*}
\end{enumerate}

In general, computing the Leray-Schauder degree for a given mapping can be challenging. Nevertheless, in many applications it can be determined explicitly. The following lemma will be used in the next section to compute the LS degree for a specific function.

\begin{lemma}
\label{computingdegree}
Let $F:X \to Y$ be $C^1$, and assume that $I - F'(x)$ is invertible for all $x \in (I-F)^{-1}(z)$. Then $(I-F)^{-1}(z)$ is finite, and
\begin{equation*}
\deg_{\operatorname{LS}}[I-F, U, z] = \sum_{x \in (I-F)^{-1}(z)} (-1)^{\sigma(x)},
\end{equation*}
where $\sigma(x)$ is the sum of the algebraic multiplicities of the eigenvalues of $F'(x)$ in $[1, \infty)$.
\end{lemma}

\subsection{Global continuation}

We are searching for solutions to the equation
\begin{equation}
\label{sitlambda}
    \left\lbrace\begin{matrix} \nabla \mathcal{B}_{\lambda}(z) = 0, \\ z\in \mathcal{Y}. \end{matrix}\right. 
\end{equation}

The first step in solving Eq.~\eqref{sitlambda} is to analyze the case $\lambda=0$, where $\beta$ is the constant introduced in \ref{max}. In this setting, explicit solutions can be found, and these solutions turn out to be isolated.

\begin{proposition}
\label{lambda0} 
Let $\beta$ be the number given in \eqref{max}, and let $\mathfrak{q}\in\mathbb{Z}^{+}$ satisfy $\mathfrak{q}>1/\sqrt{\beta}$. 
For each $\mathfrak{p} \in \left\lbrace 1, \dots, \left[ \sqrt{\beta} \, \mathfrak{q} \right] \right\rbrace$, there exists a function $w_{0} \in \mathcal{Y}$ with minimal period $2\pi \mathfrak{q}/\mathfrak{p}$ and exactly $2\mathfrak{p}$ zeros in $[0, 2\pi \mathfrak{q}]$ such that
\begin{equation*}
\nabla \mathcal{B}_{0}(w_{0}) = 0.
\end{equation*}
Moreover, there exists an open set $O \subset \mathcal{Y}$ such that
\begin{equation*}
\deg_{\operatorname{LS}}\big(\nabla \mathcal{B}_{0}, w_{0}, O \big) \neq 0.
\end{equation*}
\end{proposition}

The solutions obtained in Proposition \ref{lambda0} correspond to periodic orbits of a one-degree-of-freedom conservative system for $\lambda=0$ (the detailed proof is given in Section \ref{sec:conservative}). Our next goal is to continue these solutions along the homotopy in $\lambda$, tracking the corresponding branches. The possible behaviors of these continuation branches are illustrated in Figure \ref{figcontinuation}, and the subsequent results will be used to rule out undesired scenarios, such as branches escaping to infinity.

\begin{proposition}
\label{bounded} 
If $z = z(\cdot;\lambda) \in \mathcal{Y}$ is a solution of Eq.~\eqref{sitlambda}, then there exists a constant $M \in \mathbb{R}^{+}$, independent of $\lambda$, such that 
\begin{equation*}
\Vert z(\cdot;\lambda) \Vert_{\infty} < M.
\end{equation*}
\end{proposition}

The proof of Proposition \ref{bounded} is an adaptation of Proposition 5.1 in \cite{Ortega}. 
It relies on a comparison principle for solutions of differential inequalities, together with properties of periodic solutions of a one-degree-of-freedom conservative system and the behavior of the period function. 

To rule out intersections between distinct continuation branches, the following lemma will be useful.

\begin{lemma}
\label{zeros}
Let $z=z(\cdot;\lambda)$ be a solution of Eq.~\eqref{sitnikovhom}. Then the number of zeros of $z$ is independent of $\lambda$.
\end{lemma}

\begin{proof}
As shown in \cite{Norbert}, the number of zeros of 
$z(\cdot;\lambda) : \mathbb{R}/2\pi\mathfrak{q}\mathbb{Z} \to \mathbb{R}$ can be expressed as
\begin{equation*}
    n(\lambda) = \frac{1}{\pi} \int_{0}^{2\pi \mathfrak{q}} 
    \frac{\dot{z}^2 - \ddot{z} z}{z^2 + \dot{z}^2}  \in \mathbb{Z}.
\end{equation*}
Since $z$ depends continuously on $\lambda$, the function $n(\lambda)$ is continuous. 
Being integer-valued, it must remain constant, which proves that the number of zeros does not change with $\lambda$.
\end{proof}

Consequently, along a continuation branch, the number of zeros of a solution remains fixed. In particular, this prevents two distinct branches starting from different solutions at $\lambda=0$ from intersecting, except possibly at the trivial solution $z \equiv 0$. This observation motivates the following result, which formalizes the uniqueness of continuation branches.

\begin{proposition}
\label{unicity} 
Let $w_{1}=w_{1}(\cdot;\lambda)$ and $w_{2}=w_{2}(\cdot;\lambda)$ be two solutions of Eq.~\eqref{sitlambda} such that $w_{1}(\cdot;0) \neq w_{2}(\cdot;0)$. Then 
\begin{equation*}
w_{1}(\cdot;\lambda_{0}) = w_{2}(\cdot;\lambda_{0}) 
\end{equation*}
for some $\lambda_{0}\in [0,1]$ only if $w_{1}(\cdot;\lambda_{0}) = w_{2}(\cdot;\lambda_{0}) = 0$.
\end{proposition}

\begin{proof}
We show that if two continuation branches intersect at some $\lambda_0\in[0,1]$, then the corresponding solution must be trivial. By Proposition \ref{lambda0}, the initial solutions $w_{1}(\cdot;0)$ and $w_{2}(\cdot;0)$ have $2\mathfrak{p}_{1}$ and $2\mathfrak{p}_{2}$ zeros in $[0,2\pi\mathfrak{q}]$, respectively. Since $w_{1}(\cdot;0) \neq w_{2}(\cdot;0)$, we have $\mathfrak{p}_{1} \neq \mathfrak{p}_{2}$. Without loss of generality, assume $\mathfrak{p}_{1} > \mathfrak{p}_{2}$.

According to Lemma \ref{zeros}, the number of zeros of each solution remains constant along the homotopy. Therefore, at any $\lambda_0$, the functions $w_{1}(\cdot;\lambda_0)$ and $w_{2}(\cdot;\lambda_0)$ still have $2\mathfrak{p}_{1}$ and $2\mathfrak{p}_{2}$ zeros, respectively. If these two solutions were equal at some $\lambda_0$, $w_{1}(\cdot;\lambda_0)$ would necessarily have a zero of multiplicity at least two, i.e.,
\begin{equation}
\label{doublezero}
w_{1}(t_0;\lambda_0) = \dot{w}_{1}(t_0;\lambda_0) = 0,
\end{equation}
for some $t_0 \in [0,2\pi \mathfrak{q}]$. Since $w_1$ satisfies the second-order differential equation
\begin{equation*}
\ddot{w}_{1} = -\sum_{j=1}^{n} \frac{m_j w_1}{\big(\|H_j(t;\lambda_0)\|^2 + w_1^2\big)^{3/2}},
\end{equation*}
the Existence and Uniqueness Theorem for ODEs implies that the only solution that satisfies \eqref{doublezero} is the trivial solution
\begin{equation*}
w_1(t;\lambda_0) \equiv 0.
\end{equation*}
Hence, two distinct branches can intersect only at the trivial solution, which proves the proposition.
\end{proof}

Finally, to ensure that a continuation branch reaches the line $\lambda=1$, we construct a neighborhood around the trivial solution. The following proposition formalizes this idea.

\begin{proposition}
\label{neighborhoodaround0} 
Let $\mathfrak{p},\mathfrak{q}\in\mathbb{Z}^{+}$ satisfy
\begin{equation}  
\label{condition2}
\left( \frac{\mathfrak{p}}{\mathfrak{q}}\right)^2 \notin [\beta+1, \alpha+1],
\end{equation}
where $\alpha$ and $\beta$ are the constants defined in \eqref{max}. Then, there exists a neighborhood of $z=0$ in $\mathcal{Y}$ that contains no solutions of Eq.~\eqref{sitlambda} with exactly $\mathfrak{p}$ zeros in $[0,\pi \mathfrak{q}]$.
\end{proposition}

The proof of Proposition \ref{neighborhoodaround0} is postponed to Section 5.

\subsection{Proof of Theorem \ref{main}}

As discussed in Sect.~\ref{subsec:variational}, solutions of Eq.~\eqref{sitnikov} correspond to solutions of Eq.~\eqref{sitlambda} with $\lambda=1$. Let $\mathfrak{p} \in \{1,\dots, [\sqrt{\beta} \mathfrak{q}]\}$. Then $\mathfrak{p}/\mathfrak{q} \le \sqrt{\beta}$, so that $\mathfrak{p}$ and $\mathfrak{q}$ satisfy \eqref{condition2}. By Proposition \ref{lambda0}, there exists a function $w_{0} \in \mathcal{Y}$ with minimal period $2\pi \mathfrak{q}/\mathfrak{p}$ and exactly $2\mathfrak{p}$ zeros in $[0,2\pi \mathfrak{q}]$ such that
\begin{equation*}
\nabla \mathcal{B}_0(w_{0}) = 0.
\end{equation*}

Consider the set
\begin{equation*}
S = \{ (z,\lambda) \in \mathcal{Y} \times [0,\infty[ \ : \nabla \mathcal{B}_\lambda(z) = 0 \},
\end{equation*}
and let $\Lambda$ denote the connected component of $S$ containing $(w_{0},0)$. Our goal is to show that
\begin{equation}
\label{laramallega}
\Lambda \cap \{ \lambda = 1 \} \neq \emptyset.
\end{equation}

Assume, by contradiction, that \eqref{laramallega} does not hold. Then, for any point $(z,\lambda) \in \Lambda$, we have $\lambda < 1$, and Proposition \ref{bounded} implies that $\Vert z \Vert_{\infty} < M$. Hence, $\Lambda$ is bounded. Moreover, since $\nabla \mathcal{B}_\lambda = I - K_\lambda$ by Lemma \ref{fredholmap}, it follows that $\Lambda$ is compact. 

By Lemma 5.1 in Section 2.5 of \cite{ize}, there exists an open and bounded set $\Omega$ such that $\Lambda \subset \Omega$ and $S \cap \partial \Omega = \emptyset$. Using Proposition \ref{unicity}, we can choose $\Omega$ such that  $w_{0}$ is the unique critical point of $\mathcal{B}_{0}$ in $\Omega$. Furthermore, by Proposition \ref{neighborhoodaround0}, there exists $\varepsilon_1 > 0$ such that 
\begin{equation*}
\Omega \cap \{ \Vert z \Vert \le \varepsilon_1 \} = \emptyset.
\end{equation*}

Define the map $F : \Omega \times [0,1] \to \mathcal{Y} \times [0,1]$ such that
\begin{equation*}
F(z,\lambda;\tau) = (\nabla \mathcal{B}_\lambda(z), \lambda - \tau).
\end{equation*}
Note that $F(z,\lambda;\tau) = 0$ if and only if
\begin{equation*}
\nabla \mathcal{B}_\lambda(z) = 0 \quad \text{and} \quad \lambda = \tau.
\end{equation*}
By construction, $\nabla \mathcal{B}_\lambda(z) \neq 0$ for all $(z,\lambda) \in \partial \Omega$. Hence, by the homotopy invariance of the Leray-Schauder degree,
\begin{equation*}
\deg_{\operatorname{LS}}(F(\cdot,\cdot,0), \Omega, 0) = \deg_{\operatorname{LS}}(F(\cdot,\cdot,1), \Omega, 0).
\end{equation*}
Since $\lambda \neq 1$ in $\Omega$, the existence property of the degree implies
\begin{equation*}
\deg_{\operatorname{LS}}(F(\cdot,\cdot,1), \Omega, 0) = 0.
\end{equation*}
On the other hand, for $\tau=0$, the only solution of $F(z,\lambda;0)=0$ is $(w_{0},0)$. The excision property and Proposition \ref{lambda0} then yield
\begin{equation*}
\deg_{\operatorname{LS}}(F(\cdot,\cdot,0), \Omega, 0) = \deg_{\operatorname{LS}}(\nabla \mathcal{B}_0, O, w_{0}) \neq 0,
\end{equation*}
which is a contradiction. Therefore, \eqref{laramallega} holds.

Let $(\tilde{z},1) \in \Lambda \cap \{ \lambda = 1 \}$. Then
\begin{equation*}
\nabla \mathcal{B}_1(\tilde{z}) = 0,
\end{equation*}
so that $\tilde{z}$ is a solution of Eq.~\eqref{sitnikov}. Moreover, since $\tilde{z}$ and $w_{0}$ lie in the same connected component, $\tilde{z}$ has $2\mathfrak{p}$ zeros in $[0,2\pi \mathfrak{q}]$ (see Lemma \ref{zeros}). Hence, $\tilde{z}$ is the desired solution.

Finally, by construction, the solution $\tilde{z}$ inherits all the properties required in Theorem~\ref{main}: it is a $2\pi \mathfrak{q}$-periodic solution of Eq.~\eqref{sitnikov} with exactly $2\mathfrak{p}$ zeros in $[0,2\pi \mathfrak{q}]$, and it satisfies the symmetry properties imposed by the subspace $\mathcal{Y}$. This completes the proof of Theorem~\ref{main}. \hfill \(\square\)

\section{The conservative case}
\label{sec:conservative}

In this section, we prove Proposition \ref{lambda0}. When $\lambda=0$, the action functional reduces to
\begin{equation}  
\label{actionham}
    \mathcal{B}_{0}(z) = \int_{0}^{2\pi \mathfrak{q}} \left[ \frac{1}{2} \dot{z}(t)^2 - U_{0}(z(t)) \right] \, dt,
\end{equation}
where the potential is defined by
\begin{equation*}
    U_{0}(z) = -\sum_{j=1}^{n} \frac{m_j}{\sqrt{z^2 + \beta_j^2}}.
\end{equation*}

If $z \in \mathcal{Y}$ is a critical point of \eqref{actionham}, then by standard regularity results for weak solutions we have $z \in C^2$ and $z$ is $2\pi \mathfrak{q}$-periodic (see \cite{brezis} for details), so that the weak derivatives coincide with the usual derivatives. Consequently, critical points of \eqref{actionham} satisfy the boundary value problem
\begin{equation}  
\label{aux}
\begin{gathered}
    \ddot{z} = -U_{0}'(z),\\
    z(0) = z(2\pi\mathfrak{q}), \quad \dot{z}(0) = \dot{z}(2\pi\mathfrak{q}) = 0.
\end{gathered}
\end{equation}
Furthermore, a direct computation shows that the energy function
\begin{equation}
\label{hamiltonian}
    E = \frac{1}{2} \dot{z}^2 + U_{0}(z)
\end{equation}
is conserved along solutions of \eqref{aux}.

The problem \eqref{aux} is a one-degree-of-freedom conservative system, so that its solutions can be analyzed using classical phase-plane techniques. We can explicitly identify solutions with a prescribed number of zeros by studying the energy levels and the monotonicity of the period with respect to the energy. The next subsection is devoted to a detailed study of the period function, which will allow us to characterize these solutions more precisely.

\subsection{The period function}

We now study the period function associated with solutions of Eq.~\eqref{aux}. We focus on solutions satisfying $z(0) = \zeta > 0$ and $\dot{z}(0) = 0$. Using the energy relation \eqref{hamiltonian}, the period $T$ and the initial condition $\zeta$ are connected via the energy level $E$ as
\begin{equation}
\label{relationhyxi}
    \zeta (E)=U_{0}^{-1}(E); \qquad T(E)=\frac{4}{\sqrt{2}}\displaystyle\int_{0}^{\zeta (E)} \dfrac{1}{\sqrt{E-U_{0}(z)}} \ \mathrm{d}z.
\end{equation}

The following properties of the period function $T(E)$, which are proved in Theorem 5 of \cite{beltritti1}, will be crucial for the subsequent construction of solutions.

\begin{lemma}
\label{properties}
The period function $T = T(E)$ satisfies:
\begin{enumerate}
    \item $T$ is continuous in $E$.
    \item $T$ is strictly increasing in $E$.
    \item There exists an energy level $E_{\min} \in \mathbb{R}$ such that $\lim_{E \to E_{\min}} T(E) = 2\pi/\sqrt{\beta}$, where $\beta$ is given in Eq.~\eqref{max}.
    \item $\lim_{E \to \infty} T(E) = \infty$.
\end{enumerate}
\end{lemma}

For each energy level $E > E_0$, let $w = w(t;E)$ denote the corresponding $2\pi\mathfrak{q}$-periodic solution of Eq.~\eqref{aux}. To analyze the non-degeneracy of the critical points of the action functional $\mathcal{B}_0$, we consider the associated variational equation
\begin{equation} 
\label{variationalenergy}
    \ddot{y} = -U_{0}''\big(w(t;E)\big) \, y.
\end{equation}
The properties of the solutions of Eq.~\eqref{variationalenergy} will allow us to show that the Hessian map $\operatorname{D}^{2}\mathcal{B}_{0}(w(t;E_{0}))$ is non-degenerate and thus to compute the Leray-Schauder degree using Lemma~\ref{computingdegree}.

\begin{lemma}
\label{kernel} 
Let $E_0$ be the energy level corresponding to a $2\pi\mathfrak{q}$-periodic solution of Eq.~\eqref{aux}. Then, the dimension of the space of $2\pi \mathfrak{q}$-periodic solutions of Eq.~\eqref{variationalenergy} when $E=E_0$ is $1$.
\end{lemma}

\begin{proof}
Let $\mathcal{M}$ be the linear space of $2\pi\mathfrak{q}$-periodic solutions of Eq.~\eqref{variationalenergy} at $E=E_0$. Then, $0 \le \dim \mathcal{M} \le 2$. By direct computation, two linearly independent solutions of Eq.~\eqref{variationalenergy} are
\begin{equation*}
    y_1(t) = \left. \frac{\partial}{\partial t} w(t;E) \right|_{E=E_0}, \qquad 
    y_2(t) = \left. \frac{\partial}{\partial E} w(t;E) \right|_{E=E_0}.
\end{equation*}
The function $y_1$ clearly has period $2\pi\mathfrak{q}$, so $y_1 \in \mathcal{M}$ and hence $\dim \mathcal{M} \ge 1$.  

On the other hand, using the relation between the period and the energy given by \eqref{relationhyxi}, we can define the $2\pi\mathfrak{q}$-periodic function $\tilde{w}$ by
\begin{equation*}
    \tilde{w}(t;E) = w\Big( \frac{T(E)}{2\pi\mathfrak{q}} t; E \Big).
\end{equation*}
Differentiating with respect to $E$ at $E = E_0$ gives
\begin{equation*}
    y_2(t) = \left. \frac{\partial}{\partial E} \Big[ \tilde{w}\Big(\frac{2\pi\mathfrak{q}}{T(E)} t; E \Big) \Big] \right|_{E=E_0} 
    = - 2\pi\mathfrak{q}t\frac{T'(E_0)}{T(E_{0})^{2}}\frac{\partial \tilde{w}}{\partial t}(t; E_0)  + \frac{\partial \tilde{w}}{\partial E}(t; E_0).
\end{equation*}
By Lemma~\ref{properties}, we have $T'(E_{0})>0$. Consequently, the leading term of $y_{2}$ does not vanish, which implies that $y_{2}$ is not periodic. Hence, $y_{2}\notin\mathcal{M}$, and we conclude that $\dim \mathcal{M}=1$.
\end{proof}

\subsection{Proof of Proposition \ref{lambda0}}

To construct a $2\pi \mathfrak{q}/\mathfrak{p}$-periodic solution of Eq.~\eqref{aux}, we use the properties of the period function established in Lemma~\ref{properties}. In particular, by Point 3 of that lemma we know that $T(E)\geq 2\pi/\sqrt{\beta}$ for all $E \in \ ]E_{\min},\infty[$, and by Point 2, $T(E)$ is strictly increasing from $2\pi/\sqrt{\beta}$ to $+\infty$. 

Hence, by continuity of the period function, a $2\pi \mathfrak{q}/\mathfrak{p}$-periodic solution exists only if
\begin{equation*}
    \frac{2\pi \mathfrak{q}}{\mathfrak{p}}\geq \frac{2\pi}{\sqrt{\beta}} \quad \Longrightarrow \quad \mathfrak{p}\leq \sqrt{\beta}\,\mathfrak{q}.
\end{equation*}
Under this assumption, there exists an energy level $E_0 \in \ ]E_{\min},\infty[$ such that $T(E_0)=2\pi \mathfrak{q}/\mathfrak{p}$. The corresponding solution $w_{0} = w(\cdot;E_0)$ is then an even, $2\pi \mathfrak{q}$-periodic solution of Eq.~\eqref{aux} and, by construction, belongs to $\mathcal{Y}$. Since periodic solutions of Eq.~\eqref{aux} correspond to zeros of $\nabla \mathcal{B}_0$, the first claim of Proposition~\ref{lambda0} follows. Moreover, this zero is isolated, so there exists an open set $O \subset \mathcal{Y}$ such that $w_{0}$ is the only zero of $\nabla \mathcal{B}_0$ in $O$.

Under this assumption, there exists an energy level $E_0 \in \ ]E_{\min},\infty[$ such that $T(E_0) = 2\pi \mathfrak{q}/\mathfrak{p}$. The corresponding solution $w_{0} = w(\cdot;E_0)$ is then an even, $2\pi \mathfrak{q}$-periodic solution of Eq.~\eqref{aux} and, by construction, belongs to $\mathcal{Y}$. Since periodic solutions of Eq.~\eqref{aux} correspond to zeros of $\nabla \mathcal{B}_0$, the first claim of Proposition~\ref{lambda0} follows. Moreover, this zero is isolated, so there exists an open set $O \subset \mathcal{Y}$ such that $w_{0}$ is the only zero of $\nabla \mathcal{B}_0$ in $O$.

Next, we show that 
\begin{equation*}
    \deg_{\operatorname{LS}}(\nabla \mathcal{B}_0, w_0, O) \neq 0,
\end{equation*}
by applying Lemma~\ref{computingdegree}. To this end, we prove that the Hessian map $\operatorname{D}^{2} \mathcal{B}_0(w_0)$ is invertible. Consider the linear operator $S \in \mathcal{L}(H^1)$ that associates to any $v \in H^1$ the unique $2\pi \mathfrak{q}$-periodic solution $u$ of
\begin{equation}
\label{hessiancomplete}
\left\lbrace\begin{matrix} -\ddot{u}+ u = \big(1+U_{0}''(w_{0}(t)\big)v(t), \\ u\in H^{1}. \end{matrix}\right. 
\end{equation}
This equation is obtained from \eqref{eqhessian} by setting $\lambda = 0$ and $z = w_0$. By Lemma~\ref{fredholmmapsecond}, the Hessian map restricted to $\mathcal{Y}$ can be written as
\begin{equation*}
\operatorname{D}^{2} \mathcal{B}_0(w_0) = I - S\vert_{\mathcal{Y}}.
\end{equation*}

From Sections 2.2.6--8 of \cite{vander}, there exists an isomorphism between $\ker(I-S)$ and the space $\mathcal{M}$ defined in Lemma~\ref{kernel}, which is generated by $\dot{w}_0$. However, since $w_0$ is even, $\dot{w}_0$ is odd and therefore does not belong to $\mathcal{Y}$. It follows that $\operatorname{D}^{2} \mathcal{B}_0(w_0)$ has no zero eigenvalue and is hence invertible. Applying Lemma~\ref{computingdegree}, we conclude that
\begin{equation*}
\deg_{\operatorname{LS}}(\nabla \mathcal{B}_0, w_0, O) = (-1)^{\sigma(w_0)} \neq 0,
\end{equation*}
which completes the proof.
\hfill \(\square\)

\section{Sturm--Liouville Theory}
\label{sec:sec5}

It only remains to prove Proposition~\ref{neighborhoodaround0}, namely, to construct a neighborhood around the trivial solution that excludes certain bifurcation branches. To this end, we study the non-degeneracy of the Hessian map $\operatorname{D}^{2}\mathcal{B}_{\lambda}(0)$. Since $z=0$ is a critical point of $\mathcal{B}_{\lambda}$, Lemma \ref{fredholmmapsecond} gives
\begin{equation*}
    \operatorname{D}^{2}\mathcal{B}_{\lambda}(0) = I - L_{\lambda}(0),
\end{equation*}
where $L_{\lambda}(0)\in\mathcal{L}(\mathcal{Y})$ is a compact linear operator. It is then easy to see that $\operatorname{D}^{2}\mathcal{B}_{\lambda}(0)$ is invertible if and only if $1\notin \sigma\big(L_{\lambda}(0)\big)$, with $\sigma(L)$ denoting the spectrum of a linear operator $L$. In what follows, we will show that $1\notin \sigma\big(L_{\lambda}(0)\big)$ by applying Sturm--Liouville theory, since the property used to construct the neighborhood is the number of zeros of the solutions.

\subsection{Sturm--Liouville eigenvalue problem}

Let $z_{\mathfrak{p}}\in\mathcal{Y}$ be an eigenvector of $L_{\lambda}(0)$ with corresponding eigenvalue $\mu_{\mathfrak{p}}=\mu_{\mathfrak{p}}(\lambda)$. Then, by Eq.~\eqref{eqhessian}, $z_{\mathfrak{p}}$ satisfies
\begin{equation}
\label{SLsystem}
\begin{cases}
-\ddot{z} + z = \eta\, F_{\lambda}(t)\, z, \\[0.5em]
\dot{z}(0)=\dot{z}(\pi\mathfrak{q})=0,
\end{cases}
\end{equation}
with $\eta=\eta_{\mathfrak{p}}(\lambda)=1/\mu_{\mathfrak{p}}(\lambda)$ and
\begin{equation*}
    F_{\lambda}(t)=\sum_{j=1}^{n}\frac{m_{j}}{\Vert H_{j}(t;\lambda) \Vert^{3}}+1.
\end{equation*}
Equation \eqref{SLsystem} is a Sturm--Liouville eigenvalue problem with spectral parameter $\eta$ (see \cite[\S 27]{Walter} for the precise definition). We emphasize that two notions of eigenvalue are at play: $\mu$ is an eigenvalue of the compact operator $L_{\lambda}(0)$, i.e.\ $\mu \in \sigma\big(L_{\lambda}(0)\big)$, while $\eta$ is the eigenvalue in the Sturm--Liouville problem \eqref{SLsystem}.

By the Existence Theorem in \cite[II \S 27]{Walter}, and using the correspondence between the spectra of \eqref{SLsystem} and $L_{\lambda}(0)$, the eigenvalues of $L_{\lambda}(0)$ can be arranged as
\begin{equation*}
    \mu_{0}(\lambda)>\mu_{1}(\lambda)>\mu_{2}(\lambda)>\dots, 
    \qquad \mu_{\mathfrak{p}}(\lambda)\to 0 \quad \mbox{as} \quad \mathfrak{p}\to\infty. 
\end{equation*}
Moreover, the eigenfunction $z_{\mathfrak{p}}$ corresponding to $\mu_{\mathfrak{p}}(\lambda)$ has exactly $\mathfrak{p}$ zeros in the open interval $\, ]0,\pi\mathfrak{q}[\,$. Our goal is to exclude the possibility that $1$ belongs to the spectrum of $L_{\lambda}(0)$.

In light of the Sturm--Liouville characterization above, this amounts to proving that $\eta_{\mathfrak{p}}(\lambda)\neq 1$ for the relevant eigenvalues. Since the oscillatory behavior of eigenvectors is directly tied to the index $\mathfrak{p}$, it will be possible to exclude certain bifurcation branches by controlling the number of zeros of the corresponding eigenvectors. In preparation for the next lemma, we introduce the bounds
\begin{equation}
\label{boundsforF}
    m \coloneqq  \inf_{\lambda\in[0,1] }\inf_{t\in[0,2\pi]}F_{\lambda}(t),
    \qquad 
    M \coloneqq \sup_{\lambda\in[0,1] }\sup_{t\in[0,2\pi]}F_{\lambda}(t).
\end{equation}

\begin{lemma}
\label{branchlemma} 
Let $\mu_{\mathfrak{p}}(\lambda)$ be the eigenvalues of the operator $L_{\lambda}(0)$ and let $m$ and $M$ be the numbers defined above. If $(\mathfrak{p}/\mathfrak{q})^{2}\notin [m,M]$, then $\mu_{\mathfrak{p}}(\lambda)\neq 1$ for all $\lambda\in[0,1]$.
\end{lemma}

\begin{proof}
As a preliminary step, let us compute explicitly the spectrum of a simplified Sturm--Liouville problem with constant weight $A>0$: 
\begin{equation}
\label{simpSL}
\begin{cases}
-\ddot{z} + z = \eta A\, z, \\[0.5em]
\dot{z}(0)=\dot{z}(\pi\mathfrak{q})=0.
\end{cases}
\end{equation}
A direct computation shows that the corresponding eigenvalues are
\begin{equation*}
    \eta_{\mathfrak{p}}^{(A)} = \frac{1}{A}\left( 1+\left( \frac{\mathfrak{p}}{\mathfrak{q}} \right)^{2}\right),
    \qquad \mathfrak{p}\in\mathbb{N}.
\end{equation*}

Now, since $m<F_{\lambda}(t)<M$, we can apply the Comparison Theorem of Eigenvalues (see \cite[IX, \S 27]{Walter}), comparing system \eqref{SLsystem} with \eqref{simpSL} for $A=m$ and $A=M$. This yields the estimate
\begin{equation*}
\frac{1}{M}\left(1+\left( \frac{\mathfrak{p}}{\mathfrak{q}}\right)^{2} \right)
   \leq \eta_{\mathfrak{p}}(\lambda) 
   \leq \frac{1}{m}\left(1+\left( \frac{\mathfrak{p}}{\mathfrak{q}}\right)^{2} \right). 
\end{equation*}
Since $(\mathfrak{p}/\mathfrak{q})^{2}\notin [m,M]$ and $\mu_{\mathfrak{p}}(\lambda)=1/\eta_{\mathfrak{p}}(\lambda)$, we obtain
\begin{equation*}
\frac{m}{m+1}\leq \mu_{\mathfrak{p}}(\lambda) \leq \frac{M}{M+1}.
\end{equation*}
In particular, $\mu_{\mathfrak{p}}(\lambda)<1$ for all $\lambda\in[0,1]$, and the claim follows.
\end{proof}

\subsection{Proof of Proposition \ref{neighborhoodaround0}}

We proceed by contradiction. Assume that for every neighborhood of $z=0$ in $\mathcal{Y}$ there exists a nontrivial solution of \eqref{sitlambda} with exactly $\mathfrak{p}$ zeros in $[0,\pi\mathfrak{q}]$. Then we can construct sequences $\{z^{(i)}\}_{i\in\mathbb{N}} \subset \mathcal{Y}$, and $\{\lambda^{(i)}\}_{i\in\mathbb{N}} \subset [0,1],$ such that each $z^{(i)}$ solves \eqref{sitlambda} with parameter $\lambda^{(i)}$, $z^{(i)}$ has exactly $\mathfrak{p}$ zeros in $[0,\pi\mathfrak{q}]$, and 
\begin{equation*}
    z^{(i)} \xrightarrow[i\to\infty]{} 0 \quad \text{in } \mathcal{Y}, 
    \qquad 
    \lambda^{(i)} \xrightarrow[i\to\infty]{} \lambda^{*}\in[0,1].
\end{equation*}

Using that $z^{(i)}$ solves \eqref{sitlambda}, Lemma~\ref{fredholmap} yields 
\begin{equation*}
\nabla\mathcal{B}_{\lambda^{(i)}}(z^{(i)}) 
= (I-K_{\lambda^{(i)}})(z^{(i)}) = 0,
\qquad\text{hence}\qquad 
z^{(i)} = K_{\lambda^{(i)}}(z^{(i)}).
\end{equation*}
Consider now the normalized sequence
\begin{equation*}
\left\{\frac{z^{(i)}}{\Vert z^{(i)}\Vert_{H^{1}}}\right\}_{i\in\mathbb{N}} \subset \mathcal{Y},
\end{equation*}
which is bounded in $\mathcal{Y}$. Since $K_{\lambda}$ is compact for every $\lambda\in[0,1]$, this normalized sequence admits a convergent subsequence. That is, there exist indices $\{i_k\}_{k\in\mathbb{N}}$ and some $z^{*}\in\mathcal{Y}$ with $\|z^*\|_{H^1}=1$ such that 
\begin{equation*}
    \frac{z^{(i_k)}}{\Vert z^{(i_k)}\Vert_{H^{1}}}
    \xrightarrow[k\to\infty]{} z^{*}.
\end{equation*}

Now, since $\nabla\mathcal{B}_{\lambda}$ belongs to the class $C^{1}$, we can decompose it as 
\begin{equation*}
    \nabla\mathcal{B}_{\lambda}(z)
    = \operatorname{D}^{2}\mathcal{B}_{\lambda}(0)\,z + V_{\lambda}(z),
\end{equation*}
where $V_{\lambda}(z)=\mathcal{O}(\|z\|^{2})$ as $z\to 0$.  
Since each $z^{(i_k)}$ solves \eqref{sitlambda} for $\lambda=\lambda^{(i_k)}$, we obtain 
\begin{equation*}
    \nabla \mathcal{B}_{\lambda^{(i_k)}}(z^{(i_k)})
    = \operatorname{D}^{2}\mathcal{B}_{\lambda^{(i_k)}}(0)\,z^{(i_k)} 
      + V_{\lambda^{(i_k)}}\!\big(z^{(i_k)}\big) 
    = 0. 
\end{equation*}
Dividing both sides by $\|z^{(i_k)}\|_{H^{1}}$ and setting $v^{(k)} := \frac{z^{(i_k)}}{\|z^{(i_k)}\|_{H^{1}}}$, we obtain
\begin{equation*}
    \operatorname{D}^{2}\mathcal{B}_{\lambda^{(i_k)}}(0)\,v^{(k)} 
    + \frac{1}{\|z^{(i_k)}\|_{H^{1}}}\,V_{\lambda^{(i_k)}}\!\big(z^{(i_k)}\big)
    = 0. 
\end{equation*}
Letting $k\to\infty$ and using continuity with respect to parameters, we deduce that the limit $z^{*}$ satisfies
\begin{equation*}
   \operatorname{D}^{2}\mathcal{B}_{\lambda^{*}}(0)\,z^{*}=0 
   \qquad \Longrightarrow \qquad 
   L_{\lambda^{*}}(0)(z^{*})=z^{*}. 
\end{equation*}
Finally, since $z^{*}$ is the limit of a sequence of functions each having exactly $\mathfrak{p}$ zeros, the continuity with respect to parameters implies that 
\begin{equation}
\label{firststep}
    \mu_{\mathfrak{p}}(\lambda^{*})=1.
\end{equation}

On the other hand, formula \eqref{homotopy} implies that
\begin{equation*}
    0 < \alpha_{j} \leq \|H_{j}(t;\lambda)\| \leq \beta_{j}, 
    \qquad j=1,\ldots,n. 
\end{equation*}
Consequently,
\begin{equation*}
    \beta \leq \sum_{j=1}^{n}\frac{m_{j}}{\|H_{j}(t;\lambda)\|^{3}} 
    \leq \alpha. 
\end{equation*}
Since condition~\eqref{condition2} is satisfied, the above inequality yields
\begin{equation*}
    \left(\frac{\mathfrak{p}}{\mathfrak{q}}\right)^{2}\notin [m,M],
\end{equation*}
where $m$ and $M$ are the bounds given in \eqref{boundsforF}.  
By Lemma~\ref{branchlemma}, this implies that $\mu_{\mathfrak{p}}(\lambda^{*})\neq 0$ for all $\lambda\in[0,1]$, which contradicts \eqref{firststep}.
\hfill \(\square\)

\bibliographystyle{acm}
\bibliography{ref}

\begin{thebibliography}{10}

\bibitem{bakker}
{\sc Bakker, L., and Simmons, S.}
\newblock A separating surface for {S}itnikov-like {$n+1$}-body problems.
\newblock {\em J. Differential Equations 258}, 9 (2015), 3063--3087.

\bibitem{barrera}
{\sc Barrera, C., Bengochea, A., and Garc\'ia-Azpeitia, C.}
\newblock Comet and moon solutions in the time-dependent restricted
  {$(n+1)$}-body problem.
\newblock {\em J. Dynam. Differential Equations 34}, 2 (2022), 1187--1207.

\bibitem{beltritti2}
{\sc Beltritti, G.}
\newblock Periodic solutions of a generalized {S}itnikov problem.
\newblock {\em Celestial Mech. Dynam. Astronom. 133}, 2 (2021), Paper No. 6,
  23.

\bibitem{beltritti3}
{\sc Beltritti, G.}
\newblock On the global families of periodic solutions of a generalized
  {S}itnikov problem.
\newblock {\em Celestial Mech. Dynam. Astronom. 134}, 2 (2022), Paper No. 18,
  22.

\bibitem{beltritti1}
{\sc Beltritti, G., Mazzone, F., and Oviedo, M.}
\newblock The {S}itnikov problem for several primary bodies configurations.
\newblock {\em Celestial Mech. Dynam. Astronom. 130}, 7 (2018), Paper No. 45,
  17.

\bibitem{bountis}
{\sc Bountis, T., and Papadakis, K.~E.}
\newblock The stability of vertical motion in the {$N$}-body circular
  {S}itnikov problem.
\newblock {\em Celestial Mech. Dynam. Astronom. 104}, 1-2 (2009), 205--225.

\bibitem{brezis}
{\sc Brezis, H.}
\newblock {\em Functional Analysis, Sobolev Spaces and Partial Differential
  Equations}, 1~ed.
\newblock Universitext. Springer-Verlag New York, 2010.

\bibitem{calleja}
{\sc Calleja, R., Doedel, E., and Garc\'ia-Azpeitia, C.}
\newblock Symmetries and choreographies in families that bifurcate from the
  polygonal relative equilibrium of the {$n$}-body problem.
\newblock {\em Celestial Mech. Dynam. Astronom. 130}, 7 (2018), Paper No. 48,
  28.

\bibitem{Norbert}
{\sc Hungerb\"{u}hler, N., and Wasem, M.}
\newblock An integral that counts the zeros of a function.
\newblock {\em Open Math. 16}, 1 (2018), 1621--1633.

\bibitem{ize}
{\sc Ize, J., and Vignoli, A.}
\newblock {\em Equivariant degree theory}.
\newblock De Gruyter series in nonlinear analysis and applications 8. Walter de
  Gruyter, 2003.

\bibitem{lacomba}
{\sc Lacomba, E.~A., Llibre, J., and Perez-Chavela, E.}
\newblock The generalized {S}itnikov problem.
\newblock In {\em Celestial mechanics ({E}vanston, {IL}, 1999)}, vol.~292 of
  {\em Contemp. Math.} Amer. Math. Soc., Providence, RI, 2002, pp.~147--158.

\bibitem{li}
{\sc Li, F., Zhang, S., and Zhao, X.}
\newblock The characterization of the variational minimizers for spatial
  restricted n + 1-body problems.
\newblock {\em Abstract and Applied Analysis 2013}, 1 (2013), Article ID
  845795.

\bibitem{Ortega}
{\sc Llibre, J., and Ortega, R.}
\newblock On the {F}amilies of {P}eriodic {O}rbits of the {S}itnikov {P}roblem.
\newblock {\em SIAM J. Appl. Dyn. Syst. 7}, 2 (2008), 561--576.

\bibitem{marchesin}
{\sc Marchesin, M., and Vidal, C.}
\newblock Spatial restricted rhomboidal five-body problem and horizontal
  stability of its periodic solutions.
\newblock {\em Celestial Mech. Dynam. Astronom. 115}, 3 (2013), 261--279.

\bibitem{mawhin2}
{\sc Mawhin, J.}
\newblock Leray-{S}chauder degree: a half-century of extensions and
  applications.
\newblock {\em Topol. Methods Nonlinear Anal. 14}, 2 (1999), 195--228.

\bibitem{Mawhin}
{\sc Mawhin, J., and Willem, M.}
\newblock {\em Critical Point Theory and Hamiltonian Systems}, 1~ed.
\newblock Applied Mathematical Sciences 74. Springer-Verlag New York, 1989.

\bibitem{Palais}
{\sc Palais, R.~S.}
\newblock The {P}rinciple of {S}ymmetric {C}riticality.
\newblock {\em Comm. Math. Phys. 69}, 1 (1979), 19--30.

\bibitem{pandey}
{\sc Pandey, L.~P., and Ahmad, I.}
\newblock Periodic orbits and bifurcations in the {S}itnikov four-body problem
  when all primaries are oblate.
\newblock {\em Astrophys. Space Sci. 345}, 1 (2013), 73--83.

\bibitem{pustylnikov}
{\sc Pustyl'nikov, L.}
\newblock On certain final motions in the n-body problem.
\newblock {\em J. Appl. Math. Mech. 54}, 2 (1990), 272--274.

\bibitem{rivera1}
{\sc Rivera, A.}
\newblock Periodic {S}olutions in the {G}eneralized {S}itnikov {$(N+1)$}-{B}ody
  {P}roblem.
\newblock {\em SIAM J. Appl. Dyn. Syst. 12}, 3 (2013), 1515--1540.

\bibitem{rivera2}
{\sc Rivera, A., Perdomo, O., and Casta\~neda, N.}
\newblock Periodic oscillations in the restricted hip-hop {$2N+1$}-body
  problem.
\newblock {\em Discrete Contin. Dyn. Syst. Ser. B 28}, 11 (2023), 5481--5493.

\bibitem{Sitnikov}
{\sc Sitnikov, K.}
\newblock The existence of oscillatory motions in the three-body problems.
\newblock {\em Soviet Physics. Dokl. 5\/} (1960), 647--650.

\bibitem{soulis}
{\sc Soulis, P., Bountis, T., and Dvorak, R.}
\newblock Stability of motion in the {S}itnikov 3-body problem.
\newblock {\em Celestial Mech. Dynam. Astronom. 99}, 2 (2007), 129--148.

\bibitem{vander}
{\sc Vanderbauwhede, A.}
\newblock {\em Local bifurcation and symmetry}.
\newblock Research notes in mathematics volume 75. Pitman Advanced Publishing
  Program, 1982.

\bibitem{Walter}
{\sc Walter, W.}
\newblock {\em Ordinary Differential Equations}, 1998~ed.
\newblock Graduate Texts in Mathematics. Springer, 1998.

\end{thebibliography}

\end{document}